\documentclass[11pt,reqno]{amsart}
\usepackage{}

\def\subjclass#1{{\renewcommand{\thefootnote}{}%
\footnote{\emph{Mathematics Subject Classification (2010):} #1}}}

\usepackage{mathrsfs}
\usepackage{amsfonts}
\usepackage{amssymb,bm}
\DeclareMathOperator{\curl}{curl}
\DeclareMathOperator{\dv}{div}

\date{\today}

\hoffset=- 2cm \voffset=- 0cm 
 \theoremstyle{plain}
\newtheorem{Thm}{Theorem}

\newtheorem{Rem}[Thm]{Remark}
\newtheorem{Lem}[Thm]{Lemma}
\newtheorem{Cor}[Thm]{Corollary}

\newtheorem{Def}[Thm]{Definition}

\usepackage{amssymb}
\usepackage[active]{srcltx}
\usepackage{color}

\newcommand {\p}{\partial}
\newcommand{\q}{\quad}

\def\lam{\lambda}

\def\O{\Omega}

\def\H{\mathbf H}

\def\u{\mathbf u}
\def\w{\mathbf w}

\def\v{\vskip}

\errorcontextlines=0 \numberwithin{equation}{section}
\numberwithin{Thm}{section}

\begin{document}
\large
%Topmatter

\title[Estimate for a Maxwell-type system in convex domains]
{On an $H^r(\curl,\O)$ estimate for a Maxwell-type system in convex domains}

\author[]{Xingfei Xiang}

\address{School of Mathematical Sciences,
Tongji University, Shanghai 200092, P.R. China}
\email{xfxiang@tongji.edu.cn}

\thanks{ }


\begin{abstract}

In bounded convex domains,
the regularity estimates of a vector field $\u$
with its $\dv\u$, $\curl\u$ in $L^r$ space and
the tangential components
or the normal component of $\u$ over the
boundary in $L^r$ space, are established for $1<r<\infty$.
As an application, we derive an $H^r(\curl,\O)$ estimate
for solutions to a 
Maxwell-type system with an inhomogeneous  boundary condition
in convex domains.
\end{abstract}
\maketitle

\section{Introduction}\label{section1}

This paper is concerned with the regularity of a vector field $\u$
with its $\dv\u, \curl\u\in L^r(\O)$ and the tangential
components $\nu\times\u$
or the normal component $\nu\cdot\u$ on
boundary in $L^r(\p\O)$,
where $1<r<\infty$,  $\O$ is a bounded convex domain in $\mathbb{R}^3$
and
$\nu(x)$  denotes the unit outer normal vector at $x\in\p\O$.
Based on the established estimates,
we then study the well-posedness of the following Maxwell-type system
\begin{equation}\label{1.10}
\curl (A(x)\curl\u)+\u=\mathbf F+\curl\mathbf f\q
\text{in }\O,\q \nu\times\u=\mathbf g\q \text{on }\p\O,
\end{equation}
where the coefficient
$A(x) = (a^{ij} (x))$ denotes a $3\times3$ matrix
with real-valued, bounded, measurable
entries satisfying the uniform ellipticity condition
$$
\lambda |\xi|^2 \leq \sum_{i,j=1}^3 a^{ij}
\xi_i \xi_j \leq\Lambda |\xi|^2
$$
for all $\xi\in\mathbb R^3$ and for some
positive constants $0<\lam<\Lambda<\infty$.

 Before stating  our main results we would like to mention that,
the regularity estimates of a vector
field $\u$ by means of $\dv\u$ and $\curl\u$ are fundamental
questions, and such estimates are useful in the study of various
partial differential systems including Navier-Stokes equations in
fluid mechanics, Maxwell's equations in electromagnetism
field, and Ginzburg-Landau system for superconductivity.
For smooth domains,
the estimates on Sobolev spaces $W^{1,r}$
with $1<r<\infty$ are well-known. We refer to \cite{KY1, W1} for details.

In the case of non-smooth domains, Costabel in \cite{Cos}
considered the $\dv$-$\curl$ estimates when $r=2$
in Lipschitz domains
and showed the $H^{1/2}(\O)$ regularity for vector
fields.
These results were generalized to $r\in (3/2-\epsilon, 2+\epsilon)$
with $\epsilon$ depending on the Lipschitz character
of domains
by D. Mitrea, M. Mitrea and J. Pipher (see \cite{MMP}), and also the range for $r$
is sharp (see \cite{DK, FMM}). It
should also be noted in \cite{JKD}
that if the boundary $\p\O\in C^1,$ then
one can obtain the corresponding
estimates for $r\in (1, \infty)$.
One may ask, under what additional
conditions (weaker than $C^1$ regularity)
for Lipschitz domains, the range for $r$
can be extended to  the interval $(1, \infty)$?

Note that any convex domain is Lipschitz
but may not be $C^1$, and also the convexity
of the domain may improve the regularity,
see for instance \cite{CM, CM1, Geng-Shen, MMY}.
Therefore, it is important to examine
the estimates in convex domains.
To state our results, we need to introduce
the well-known
Bessel potential spaces $L_{\alpha}^r(\O)$
and Besov spaces $B_{\alpha}^{r,q}(\O)$, see \cite{JKD}.
First, we define $L_{\alpha}^r(\mathbb R^3)$ by
$$
L_{\alpha}^r(\mathbb R^3)=\left\{(I-\Delta)^{-\alpha/2}g~:~g\in L^r(\mathbb R^3)\right\}
$$
with norm
$$
\|f\|_{L_{\alpha}^r(\mathbb R^3)}=\left\|(I-\Delta)^{\alpha/2}f\right\|_{L^r(\mathbb R^3)},
$$
where
$$
(I-\Delta)^{\alpha/2}=\mathcal {F}^{-1}\left(1+|\xi|^2\right)^{\alpha/2}\mathcal {F}
$$
and $\mathcal {F}$ is the Fourier transform.
Define  $L_{\alpha}^r(\O)$
as the space of restrictions of functions
in $L_{\alpha}^r(\mathbb R^3)$ to $\O$ with the
usual quotient norm
$$
\|f\|_{L_{\alpha}^r(\O)}=\inf\left\{\|h\|_{L_{\alpha}^r(\mathbb R^3)}~:~h=f\q \text{in }\O\right\}.
$$
Let $0<\alpha<1,$
$1\leq r\leq \infty$ and $1\leq q<\infty$. We say that
a function $f$ belongs to
 Besov space $B_{\alpha}^{r,q}(\mathbb R^3)$ if the norm
$$
\|f\|_{L^r(\mathbb R^3)}
+\left(\int_{\mathbb R^3}\frac{\|f(x+t)-f(x)\|_{L^r}^q}{|t|^{3+\alpha q}}dt\right)^{1/q}<\infty.
$$
Define the space $B_{\alpha}^{r,q}(\O)$
as the space of restrictions of functions
in $B_{\alpha}^{r,q}(\mathbb R^3)$ to $\O$
with the usual
quotient norm.

Suppose $1<r<\infty$ and $\alpha>0.$
Then we have the following inclusion relations (see Theorem 5 in \cite[Chapter V]{ST})
$$
B_{\alpha}^{r,2}\subset L_{\alpha}^r\subset B_{\alpha}^{r,r}\q \text{if }r\geq 2;\q
B_{\alpha}^{r,r}\subset L_{\alpha}^r\subset B_{\alpha}^{r,2}\q \text{if }r\leq 2.
$$

 The first result now
reads:

\begin{Thm}\label{thm1.1}
Let $\O$ be a bounded convex domain
in $\mathbb R^3$.
Assume that $\dv\u\in L^r(\O)$,
$\curl\u\in L^r(\O)$
and  $\nu\cdot\u\in L^r(\p\O)$
with $2\leq r<\infty.$ Then $\u\in L_{1/r}^r(\O),$
and we have the estimate
\begin{equation}\label{1.2}
\|\u\|_{L_{1/r}^r(\O)}\leq C\left(\|\dv\u\|_{L^r(\O)}+\|\curl\u\|_{L^r(\O)}
+\|\nu\cdot\u\|_{L^r(\p\O)}\right),
\end{equation}
where the constant $C$ depends on $r$ and
the Lipschitz character of $\O.$
\end{Thm}

To prove Theorem \ref{thm1.1}, we apply  the
Helmholtz-Weyl decomposition for vector fields
 in bounded domains
(see \cite[Theorem 2.1]{KY1}):
\begin{equation}\label{1.4}
\u=\nabla p_{\u}+\curl\mathbf w_{\u}.
\end{equation}
Our strategy is to get the estimates
for the gradient part $\nabla p_{\u}$
and for the curl part $\curl\mathbf w_{\u}$ respectively.
The gradient part $\nabla p_{\u}$ satisfies
the Laplace equation with
Neumann boundary condition, which can be established
by the result of  Geng and Shen in \cite{Geng-Shen}
for Laplace-Neumann problem.
For the estimate of $\curl\mathbf w_{\u}$,
 the vector $\mathbf w_{\u}$ satisfies 
 a curl-curl system (see \eqref{2.1}). As the proof of Theorem 5.15(a)
in \cite{JKD} by Jerison and  Kenig,
it suffices to establish the $L^{\infty}$
estimate for $\curl\mathbf w_{\u}$. To prove this,
we shall use the technique developed by Cianchi
and Maz'ya in \cite{CM, CM1}
in which the $L^{\infty}$ gradient estimates of solutions
to the divergence form elliptic systems
with Uhlenbeck type structure were treated.
At last, by the complex interpolation,
we can obtain the $L_{1/r}^r(\O)$
estimate for $\curl\mathbf w_{\u}$ if $2\leq r<\infty$.

\begin{Rem} We need to mention that
for Lipschitz domains,
D. Mitrea, M. Mitrea and J. Pipher in \cite{MMP}
obtained the $B_{1/r}^{r, 2}(\O)$ estimates
under the assumptions of Theorem \ref{thm1.1}
if
 $r\in (3/2-\epsilon, 2]$
with $\epsilon$ depending on the Lipschitz character
of  domains.
The $B_{1/r}^{r, 2}(\O)$ estimate for $r\in(1, 3/2-\epsilon]$ is still open.
\end{Rem}

For the tangential component $\nu\times\u$
given, we have

\begin{Thm}\label{thm1.2}
Let $\O$ be a bounded convex domain
in $\mathbb R^3$.
Assume that $\u\in L^r(\O)$, $\dv\u\in L^r(\O)$,
$\curl\u\in L^r(\O)$
and  $\nu\times\u\in L^r(\p\O)$
with $1<r<\infty.$ Then
\begin{equation}\label{1.1}
\|\nu\cdot\u\|_{L^r(\p\O)}\leq C\left(\|\dv\u\|_{L^r(\O)}+\|\curl\u\|_{L^r(\O)}
+\|\nu\times\u\|_{L^r(\p\O)}\right),
\end{equation}
where the constant $C$ depends on $r$ and
the Lipschitz character of $\O.$
Also, we have $\u\in L_{1/r}^r(\O)$ if $2\leq r<\infty$
and $\u\in B_{1/r}^{r, 2}(\O)$ if $1<r<2.$
\end{Thm}

To obtain the
estimate of $\nu\cdot \u$ on
boundary, the method of the complex interpolation is no longer
applied.
Our strategy now is by introducing a divergence-free vector
such  that the boundary estimate can be
reduced to the estimates of a double layer potential
and the Laplace equation with
Dirichlet boundary condition.

With Theorem \ref{thm1.1} and Theorem \ref{thm1.2}
at our disposal, following the real variable method used
in \cite{Geng} by Geng in Lipschitz domains, we then study
the $H^r(\curl, \O)$ well-posedness of the Maxwell-type system \eqref{1.10}
in convex domains.

%The well posedness of the problem (1.3) has been derived in the L2-based Sobolev
%spaces, i.e., H(curl; 次), for domains with minimal smoothness and minimum regularity
%assumptions on the coefficients, see for instance [15].
%In order to study the question of regularity for the inhomogeneous Dirichlet每Laplace
%problem, in [10], Jerison and Kenig used harmonic analysis technique to obtain a best
%possible estimate for the solutions in Sobolev每Besov L p
%s (次) norms with optimal range
%of the smoothness index s and the integrability index p.
We mention that if the coefficient matrix $A(x)$ is
taken to be a constant,  M. Mitrea, D.
Mitrea and J. Pipher in \cite{MMP} considered the $L^p$ estimates of
inhomogeneous boundary value problems
for Maxwell equations in Lipschitz
domains; while  M. Mitrea in \cite{MM}
showed the well-posedness in the Sobolev-Besov spaces $H^{s,r}_0 (curl; \O)$ with
the smoothness index $s$ and the integrability index $r$
belonging to $\mathcal {R}_{\O},$ where
$\mathcal {R}_{\O}$ defined in \cite{JKD} (also see \cite{MM} for details)
is the optimal range of solvability of Poisson
equation with inhomogeneous Dirichlet or Neumann boundary condition in
Sobolev-Besov $L_s^r(\O)$ spaces.
For system \eqref{1.10} with the $W^{s,\infty}$-regular
matrix $A(x),$
Kar and Sini in \cite{KS} recently, by the perturbation argument,
derived an $H^{s, r}_0 (curl; \O)$
estimate if the indices $(s, r)$ lie in a
small region in the interior of $\mathcal {R}_{\O}.$

In contrast to the method used in \cite{KS}, we will
apply the real variable method which was used in \cite{Geng} to treat
the $\dv(A(x)\nabla)$ operator,
 to the $\curl(A(x)\curl)$ operator.
As in \cite{Geng}, we also assume that
 the coefficient $A(x)$ belongs to $\mathrm{VMO}(\O),$ that is
$$
\lim_{r\to 0}\sup_{\rho\leq r}\frac{1}{|\O_{\rho}|}
\int_{\O_{\rho}}\Big|a^{ij}(x)-\frac{1}{|\O_{\rho}|}
\int_{\O_{\rho}}a^{ij}(y)dy\Big|dx=0,
$$
where  $\O_{\rho}$ is the intersection
$\O\bigcap B_{\rho}$ with Lebesgue
measure $|\O_{\rho}|$, and $B_{\rho}$ denotes the ball
with radius $\rho$ centered at the points of $\O$.
The following  spaces $H^{r}(\curl,\O)$ for
$1<r<\infty$ are well known:
$$
H^r(\curl,\O)=\left\{\u\in L^r(\O)~:~\curl\u\in L^r(\O)\right\}.
$$
For $1<r<\infty$ and $0<s<1$, we let $B^{s, r}(\p\O)$
denote the Besov space consisting of measurable functions on $\p\O$
such that
$$
\|f\|_{B^{s, r}(\p\O)}:=\|f\|_{L^{r}(\p\O)}+
\left(\int_{\p\O}\int_{\p\O} \frac{|f(P)-f(Q)|^r}{|P-Q|^{2+sr}}
d\sigma(P)d\sigma(Q)\right)^{1/r}<\infty,
$$
and $B^{-s, r/(r-1)}(\p\O)$ is the dual of the Besov space
$B^{s, r}(\p\O)$.
Denote by $\mathrm{Div}$
 the divergence operator on $\p\O$, the definition
 of which can be found in
\cite[p.143]{MMP}.

Now we state the $H^r(\curl,\O)$ estimate for system \eqref{1.10}.

\begin{Thm}\label{thm1.4}
Let $\O$ be a bounded convex domain
in $\mathbb R^3$.
Assume that the coefficient matrix
$A(x)$ is symmetric, bounded measurable,
uniformly elliptic and in $\mathrm{VMO}(\O).$
Let $1<r<\infty.$ Suppose that $\mathbf F\in L^r(\O)$,
$\mathbf f\in L^r(\O)$,
$\mathbf g\in L^r(\p\O)$ with $\nu\cdot\mathbf g=0$ and
$\mathrm{Div}\,\mathbf g\in B^{-1/r,r}(\p\O),$
then
there exists a unique solution $\u\in H^r(\curl,\O)$
of system \eqref{1.10}, and the solution $\u$ satisfies the estimate
\begin{equation}\label{1.5}
\aligned
&\|\mathbf u\|_{L^r(\O)}+\|\curl\mathbf u\|_{L^r(\O)}\\
\leq &C \left(\|\mathbf F\|_{L^r(\O)}+\|\mathbf f\|_{L^r(\O)}+
\|\mathrm{Div}\,\mathbf g\|_{B^{-1/r,r}(\p\O)}+\|\mathbf g\|_{L^{r}(\p\O)}\right),
\endaligned
\end{equation}
where the constant $C$ depends on $r$ and
the Lipschitz character of $\O.$ Moreover,
assume further that  $\dv\mathbf F\in L^r(\O)$, then
$\u\in L_{1/r}^r(\O)$ if $2\leq r<\infty$
and $\u\in B_{1/r}^{r, 2}(\O)$ if $1<r<2.$

\end{Thm}

\begin{Rem} Using the proof of Theorem \ref{thm1.4},
if the domain $\O$ is  Lipschitz
and  $r\in (3/2-\epsilon, 3+\epsilon)$
for some positive constant $\epsilon$ depending on the Lipschitz character of $\O$,
we can also obtain the inequality \eqref{1.5}.
This can be viewed as an improvement of the $s=0$ setting of
Kar and Sini's $H^{s, r}_0 (curl; \O)$ estimate in \cite{KS},
see Theorem \ref{thm3.2} for details.
\end{Rem}

The organization of this paper is as follows. In Section 2,
we first establish the $L^{\infty}$ estimates for
 vector fields with the
normal component or the tangential
components vanishing on the boundary.
Then we will give the  proofs of
Theorem \ref{thm1.1} and Theorem \ref{thm1.2}.
In Section 3,
applying Theorem \ref{thm1.1} and Theorem \ref{thm1.2},
we  prove Theorem \ref{thm1.4}.
At last, we show the well-posedness of the
Maxwell-type system in
Lipschitz domains.

 Throughout the paper,
 the bold typeface is used to indicate vector
quantities; normal typeface will be used for
vector components and for scalars.

\section{Proofs of Theorem \ref{thm1.1} and Theorem \ref{thm1.2}}

Consider the system
\begin{equation}\label{2.1}
\curl\curl\w=\curl\u,\q \dv\w=0 \q  \text{in }\O,\q \nu\times\w=0\q \text{on }\p\O
\end{equation}
and the system
\begin{equation}\label{2.2}
\aligned
\curl\curl\hat{\w}=\curl\u\q &\text{and}\q\dv\hat{\w}=0 \q
&\text{in }\O,\q\\
 \nu\cdot\hat{\w}=0\q
&\text{and}\q\nu\times\curl\hat{\w}=\nu\times\u\q &\text{on }\p\O.~
\endaligned
\end{equation}
To define the respective weak solutions of systems \eqref{2.1} and \eqref{2.2},
we introduce two spaces (\cite{KY1}):
$$
\aligned X_{\sigma}^r\equiv&\left\{\u\in H^{r}(\curl, \O)~:~\dv\u=0\q\text{in }\O,\q
\nu\cdot\u=0\q\text{on }\p\O\right\},\\
 V_{\sigma}^r\equiv&\left\{\u\in H^{r}(\curl, \O)~:~\dv\u=0\q\text{in }\O,
 \q\nu\times\u=0\q\text{on }\p\O\right\},
\endaligned
$$
where $1<r<\infty.$

\begin{Def}
We say $\w$ is a weak  solution to system \eqref{2.1}
if $\w\in V_{\sigma}^r$
and
$$
\int_{\O} \curl\w\cdot\curl\mathbf\Phi dx=\int_{\O} \u\cdot\curl\mathbf\Phi dx
$$
for any $\mathbf\Phi\in  V_{\sigma}^{r/(r-1)}.$

We say $\hat{\w}$ is a weak  solution to system \eqref{2.2}
if $\hat{\w}\in X_{\sigma}^r$
and
$$
\int_{\O} \curl\hat{\w}\cdot\curl\mathbf\Phi dx=\int_{\O} \u\cdot\curl\mathbf\Phi dx
$$
for any $\mathbf\Phi\in X_{\sigma}^{r/(r-1)}.$
\end{Def}

As stated in the introduction, to prove Theorem \ref{thm1.1}
by applying the complex
interpolation,  the key step is to
establish the $L^{\infty}$ estimate
 for the curl of solutions to the
 curl-type system \eqref{2.1}.
 We need to mention that the proof of $L^{\infty}$ estimate
 is inspired by Cianchi
and Maz'ya in \cite{CM, CM1}
where the divergence-type elliptic systems
with Uhlenbeck type structure were treated.

 We first establish
 an inequality for vector fields
 with the normal component or the tangential component
 vanishing in convex domains.
 A similar result can be found in
 \cite[Lemma 2.2]{Geng-Shen}.

\begin{Lem}\label{lem3.1}
Let $\O$ be a bounded convex domain in $\mathbb R^3$
with smooth boundary. Let
$\H\in C^2({\O})\bigcap C^1(\bar{\O})$
satisfying $\nu\cdot\H=0$ or $\nu\times\H=0$ on $\p\O$.
Then
\begin{equation}\label{3.3}
\aligned
\int_{\{|\H|=t\}}
t|\nabla|\H|| dS
\leq &\int_{\{|\H|=t\}}t
\left(|\curl\H|+|\dv\H|\right)dS\\
&+\int_{\{|\H|>t\}} \left(|\curl\H|^2+|\dv\H|^2\right) dx.
\endaligned
\end{equation}
\end{Lem}
\begin{proof}
We first note that
$$
\dv(\nabla |\H| |\H|)+\dv(\curl\H\times\H)
=\nabla \dv\H\cdot\H+|\nabla \H|^2
-|\curl\H|^2.
$$
Then by Green's formula, we have
$$
\aligned
&\int_{\{|\H|>t\}}\left(\dv(\nabla |\H| |\H|)+\dv(\curl\H\times\H)\right) dx\\
=&\int_{\{|\H|=t\}\bigcap\p\O}\sum_{i,j=1}^3 \nu_i H_j \p_j H_i dS
+\int_{\{|\H|=t\}\backslash\p\O}\left(\curl \H\times \H +\nabla|\H||\H|\right)\cdot \nu(x)dS
\endaligned
$$
and
$$
\aligned
&\int_{\{|\H|>t\}} \left(\nabla \dv\H\cdot\H+|\nabla \H|^2
-|\curl\H|^2\right) dx\\
=&\int_{\p\{|\H|>t\}} \nu\cdot\H\dv\H dS
+\int_{\{|\H|>t\}} \left(|\nabla \H|^2
-|\curl\H|^2-|\dv\H|^2\right) dx.
\endaligned
$$
Therefore,
$$
\aligned
&\int_{\{|\H|=t\}\bigcap\p\O}\left(\sum_{i,j=1}^3 \nu_i H_j \p_j H_i -\nu\cdot\H\dv\H\right)dS\\
=&\int_{\{|\H|=t\}\backslash\p\O}\left(\H\dv\H -\curl \H\times \H -\nabla|\H||\H|\right)\cdot \nu(x)
dS\\
&+\int_{\{|\H|>t\}} \left(|\nabla \H|^2
-|\curl\H|^2-|\dv\H|^2\right) dx.
\endaligned
$$
From \cite[p.135-137]{GP} and
by the condition $\nu\cdot\H=0$ or $\nu\times\H=0$ on $\p\O$, then it follows that
$$
\int_{\{|\H|=t\}\bigcap\p\O}\left(\sum_{i,j=1}^3 \nu_i H_j \p_j H_i -\nu\cdot\H\dv\H\right)dS\leq 0.
$$
This gives that
\begin{equation}\label{5.5}
\aligned
-\int_{\{|\H|=t\}}\nabla|\H||\H|\cdot \nu(x)dS
\leq &\int_{\{|\H|=t\}}t
\left(|\curl\H|+|\dv\H|\right)dS\\
&+\int_{\{|\H|>t\}} \left(|\curl\H|^2+|\dv\H|^2\right) dx.
\endaligned
\end{equation}
Note that, for $x\in \{|\H|=t\}\bigcap\{|\nabla |\H||\neq 0\}$ we have
$$
\nu(x)=-\frac{\nabla |\H|}{|\nabla |\H||}.
$$
From Sard's theorem, we know that
\begin{center}
the image  $|\H|(X)$ has Lebesgue measure $0$, where $X=\{|\nabla |\H||=0\}$.
\end{center}
Then, the inequality \eqref{3.3} follows since  \eqref{5.5}.
\end{proof}

To show the $L^{\infty}$
estimate for $\curl\w$  of
system \eqref{2.1}, it is necessary to introduce
the well-known Lorentz spaces.
Let  $f$ be
a measurable function defined on $\O$. We define the distribution function of $f$ as
$$f_{*}(s)=\mu(\{|f|>s\}),\q s>0,
$$
and the nonincreasing rearrangement of $f$ as
$$f^{*}(t)=\inf\{s>0, f_{*}(s)\leq t\},\q t>0.
$$
The Lorentz space is defined as
$$L^{m,q}(\O)=\left\{f: \O\to\mathbb{R} \text{ measurable},
\|f\|_{L^{m,q}(\O)}<\infty\right\} \q \text{with } 1\leq m<\infty
$$
equipped with the quasi-norm
$$\|f\|_{L^{m,q}(\O)}=\Big(\int_0^{\infty}\left(t^{1/m}f^{*}(t)\right)^q\frac{dt}{t}\Big)^{1/q},
\q 1\leq q<\infty,
$$
see for example \cite[p.223-p.228]{AD}  for a more precise definition.
Furthermore, the property that the Lebesgue space $L^{r}(\O)$ is  continuously
imbedded into  $L^{m,q}(\O)$ if $r>m$ will be used in the following proofs.

\begin{Lem}\label{lem3.2}
Let $\O$ be a bounded convex domain in $\mathbb R^3.$
Let $\u\in H^r(\curl, \O)$ with $r>3$
and let $\w$ be the weak solution of system \eqref{2.1}.
Then we have
\begin{equation}\label{3.6}
\|\curl\w\|_{L^{\infty}(\O)}\leq C \|\curl\u\|_{L^{3,1}(\O)},
\end{equation}
where the constant C depends on the Lipschitz character of the domain $\O$.
\end{Lem}

\begin{proof}
We divide the proof into three steps.

Step 1. We prove \eqref{3.6} under the following assumptions:

(i) the vector $\u\in C^3({\O})\bigcap C^2(\bar{\O})$;

(ii) the domain $\O$ is smooth.

Let $\H=\curl\w.$ From Lemma \ref{lem3.1}, we now have
\begin{equation}\label{3.7}
\int_{\{|\H|=t\}}
t|\nabla|\H|| dS
\leq \int_{\{|\H|=t\}}t
|\curl\u| dS\\
+\int_{\{|\H|>t\}} |\curl\u|^2 dx.
\end{equation}
We need to mention that
the inequality \eqref{3.7} is quite
similar to the inequality (6.16)
in \cite{CM}. Therefore,
to obtain the estimate \eqref{3.6} under the assumptions
(i) and (ii)
the proof in \cite{CM} is applicable.
For reader's convenience,
we give the outline of the proof in appendix.

Step 2. We remove the assumption (i).
We take a sequence $\u_k\in C^3(\bar{\O})$
such that $\u_k$ converges to $\u$ in $H^r(\curl,\O).$
Let $\w_k$ be the solution of system \eqref{2.1}
with $\curl\u$ replaced by $\curl\u_k.$ Then we have
$\w_k\in C^3(\bar{\O})$ and by \eqref{3.11} in appendix we have
\begin{equation}\label{3.12}
\|\curl\w_k\|_{L^{\infty}(\O)}\leq C(\O) \|\curl\u_k\|_{L^{3,1}(\O)}.
\end{equation}
From system \eqref{2.1}, we know that $\w_k\in V_{\sigma}^2$ and $\curl\w_k\in X_{\sigma}^2$.
Note that the spaces
$X_{\sigma}^2$ and $V_{\sigma}^2$  are both continuously
imbedded into  $H^1(\O)$ in convex domains
(see \cite[Theorem 2.17]{ABDG}), then
we can deduce that
$$
\|\w_k\|_{H^{1}(\O)}+\|\curl\w_k\|_{H^{1}(\O)}\leq C(\O) \|\curl\u_k\|_{L^{2}(\O)}.
$$
Then there exists a vector $\w\in H^1(\O)$ such that  $\w$ is the weak
solution of system \eqref{2.1}. Moreover, there exists
a subsequence of $\{\w_k\}_{k=1}^{\infty}$, still denoted by $\{\w_k\}_{k=1}^{\infty}$,  such that
$$
\curl\w_{k}\rightarrow \curl\w \q \text{in }L^2(\O)
$$
and
$$
\curl\w_{k}\rightarrow \curl\w \q \text{almost everywhere on } \O.
$$
From \eqref{3.12}, the solution $\w$
satisfies the estimate \eqref{3.6}.

Step 3. We remove the assumption (ii).  We look
for a sequence $\{\O_m\}_{m\in \mathbb{N}}$
of bounded domains $\O_m \subset\O$ such that
$\O_m\in C^{\infty},  \O_m\to\O$ as $m\to\infty$
with respect to the Necas-Verchota's approximation,
see \cite{Ne, VE}. Let $\w_m$ be the solution of system \eqref{2.1}
with the domain $\O$ replaced by $\O_m.$
Then by \eqref{3.11} in appendix we have
\begin{equation}\label{3.13}
\|\curl\w_m\|_{L^{\infty}(\O_m)}\leq C\|\curl\u\|_{L^{3,1}(\O)}.
\end{equation}
where the constant $C$ depends on the Lipschitz character of $\O_m,$
and hence depends on the Lipschitz character of $\O.$

From system \eqref{2.1},  we can also conclude that
$\w_m\in V_{\sigma}^2$ and $\curl\w_m\in X_{\sigma}^2$.
Then by Theorem 2.17 in \cite{ABDG} again, we have
\begin{equation}\label{3.14}
\|\w_m\|_{H^{1}(\O_m)}+\|\curl\w_m\|_{H^{1}(\O_m)}\leq C(\O) \|\curl\u\|_{L^{2}(\O)}.
\end{equation}
Let $\tilde{\w}_m$ be the extension of $\w_m$
 such that $\tilde{\w}_m$ is 0 outside of $\O_m$.
 Then we obtain that $\tilde{\w}_m$ converges to $\w$ weakly in $L^2(\O)$
 and $\curl\tilde{\w}_m$ converges to $\curl\w$ weakly in $L^2(\O),$
 where $\w\in H^1(\O)$ is the weak solution of system \eqref{2.1}.
 From \eqref{3.14}, for any compact subset $K$ of $\O$ we have
$$
\curl\w_m\rightarrow \curl\w \q \text{almost everywhere on any compact set } K.
$$
By \eqref{3.13}, the solution $\w$
satisfies the estimate \eqref{3.6}. We finish our proof. \end{proof}

By Theorem 2.2 and Theorem 2.5 in \cite{CM},
then from Lemma \ref{lem3.2} and the
Helmholtz-Weyl decomposition \eqref{1.4}, we immediately get

\begin{Cor}\label{cor2.4}
Let $\O$ be a bounded convex domain
in $\mathbb R^3$.
Let $\u\in L^{3,1}(\O),$ $\dv\u\in L^{3,1}(\O)$ and
$\curl\u\in L^{3,1}(\O)$. Assume further that $\nu\times\u=0$ or $\nu\cdot\u=0$ on $\p\O,$
then $\u\in L^{\infty}(\O)$ and we have the estimate
$$
\|\u\|_{L^{\infty}(\O)}\leq C\left(\|\dv\u\|_{L^{3,1}(\O)}+\|\curl\u\|_{L^{3,1}(\O)}
\right),
$$
where the constant $C$ depends only on
the Lipschitz character of $\O.$
\end{Cor}

Next, we prove the $L_{1/r}^r(\O)$ estimate for $\curl\w$ of system \eqref{2.1}.

\begin{Lem}
Let $\O$ be a bounded convex domain in $\mathbb R^3.$
Let $\u\in H^r(\curl, \O)$ with $r>2$
and let $\w$ be the weak solution of system \eqref{2.1}.
Then we have
\begin{equation}\label{3.15}
\|\curl\w\|_{L^{r}_{1/r}(\O)}\leq C \|\curl\u\|_{L^{r}(\O)},
\end{equation}
where the constant C depends on $r$
and the Lipschitz character of the domain $\O$.
\end{Lem}

\begin{proof}
The proof is similar to that of Theorem 5.15(a) in \cite{JKD}.
Let $\mathscr{E}$ be Stein's extension operator mapping from functions
on  $\O$ to functions on  ${\mathbb R}^3$
(see \cite{ST}).
Denote by $\Lambda^z$
the fractional integral operator
$$
\Lambda^z f=\mathcal {F}^{-1}\left(\left(1+|\xi|^2\right)^{z/2}(\mathcal {F} f(\xi))\right).
$$
 Then we define the mapping
$$
\mathscr{M}_z:\q \curl\u\mapsto \Lambda^{z}\mathscr{E}\curl\w.
$$
From Lemma \ref{lem3.2}, for $\mathrm{Re} z=0$ the
mapping $\mathscr{M}$ maps $L^{3,1}(\O)$(and hence $L^{\infty}$)$\to BMO(\mathbb R^3).$
For $\mathrm{Re} z=1,$ it maps $L^2(\O)\to L^2(\mathbb R^3).$
Therefore, by the complex interpolation, when $z=2/r$
it maps $L^r(\O)\to L^r(\mathbb R^3),$ which proves that
if $\curl\u\in L^r(\O),$ then $\curl\w\in L^{r}_{1/r}(\O).$
This shows that \eqref{3.15} holds.
\end{proof}

We now begin to prove our main theorems.

\begin{proof}[Proof of Theorem \ref{thm1.1}]
Consider the following Laplace equation with Neumann boundary condition
\begin{equation}\label{2.33}
\Delta {p}=\dv\u \q  \text{in }\O,\q \frac{\p p}{\p\nu}=\nu\cdot\u\q \text{on }\p\O.
\end{equation}\label{2.31}
Let
\begin{equation}\label{2.30}
\tilde{p}=-\frac{1}{4\pi}\int_{\O} \frac{1}{|x-y|} \dv\u(y) dy.
\end{equation}
Then the function ${p}-\tilde{p}$ satisfies
$$
\Delta ({p}-\tilde{p})=0 \q  \text{in }\O,\q
\frac{\p({p}-\tilde{p})}{\p\nu}=\nu\cdot\u
-\frac{\p\tilde{p}}{\p\nu}\q \text{on }\p\O.
$$
The solvability of the solution ${p}-\tilde{p}$ to the above equation
can be found in \cite[Theorem 1.1]{Geng-Shen}, which implies the solvability
of problem \eqref{2.33}. Moreover,  Theorem 1.1 in \cite{Geng-Shen} gives the estimate
$$
\aligned
\|\nabla ({p}-\tilde{p})\|_{L_{1/r}^{r}(\O)}
&\leq
C(r,\O)\left(\left\|\nu\cdot\u\right\|_{L^{r}(\p\O)}+\left\|\frac{\p \tilde{p}}{\p\nu}\right\|_{L^{r}(\p\O)}\right)\\
&\leq C(r,\O) \left(\left\|\nu\cdot\u\right\|_{L^{r}(\p\O)}+\|\dv\u\|_{L^{r}(\O)}\right),
\endaligned
$$
where we have used the trace theorem and
the Calderon-Zygmund inequality in the last inequality.
Applying the Calderon-Zygmund inequality again for $\tilde{p}$, we have
\begin{equation}\label{3.17}
\|\nabla {p}\|_{L_{1/r}^{r}(\O)}
\leq C(r,\O) \left(\left\|\nu\cdot\u\right\|_{L^{r}(\p\O)}
+\|\dv\u\|_{L^{r}(\O)}\right)\q \text{for } 2<r<\infty.
\end{equation}

Now we let
 $$
 \tilde{\u}=\nabla {p}+\curl {\w},
 $$
 where $p$ is defined in \eqref{2.33} and
 $\w$ is the weak solution of system \eqref{2.1}.
Then we have
 $$
\dv (\tilde{\u}-\u)=0,\, \curl (\tilde{\u}-\u)=0\,\,\,\text{in }\O,\q
\nu\cdot(\tilde{\u}-\u)=0\,\,\,\text{on }\p\O,
$$
which  shows that $\tilde{\u}=\u$ in $\O$.
Therefore,  the inequality \eqref{1.2} holds true
since \eqref{3.17} and \eqref{3.15}.
We finish our proof.
\end{proof}

We are now in the position to show Theorem \ref{thm1.2}. In the proof, we shall use
the symbol $(h)^*$ to
denote the nontangential maximal function of $h$ in $\O$,  defined as
$$
(h)^*(x)=\sup\left\{|u(y)|,y \in\O, |x-y|<2 \mathrm{dist}(y,\p\O)\right\},\q x\in\p\O;
$$
we also introduce  the tangential derivative of a function
$\psi$ defined on $\p\O$ by $\nabla_{\mathrm{tan}}\psi$, we refer to
\cite[p.2518]{MMY}
for its definition, in particular,
if $\psi$ is a Lipschitz function
then $\nabla_{\mathrm{tan}}\psi=\nu\times\nabla\psi$ almost everywhere on $\p\O$.

\begin{proof}[Proof of Theorem \ref{thm1.2}]
Let $\hat{p}$ be the weak solution of Laplace equation
$$
\Delta \hat{p}=\dv\u \q  \text{in }\O,\q \hat{p}=0\q \text{on }\p\O,
$$
and let $\tilde{p}$ be defined as \eqref{2.30}.
The function $\hat{p}-\tilde{p}$ satisfies
$$
\Delta (\hat{p}-\tilde{p})=0 \q  \text{in }\O,\q
\hat{p}-\tilde{p}=-\tilde{p}\q \text{on }\p\O.
$$
For $1<r<\infty$,  we have
$$
\|(\nabla(\hat{p}-\tilde{p}))^*\|_{L^{r}(\p\O)}
\leq C(r,\O) \left(\|\tilde{p}\|_{L^{r}(\p\O)}
+\|\nabla_{\mathrm{tan}}\tilde{p}\|_{L^{r}(\p\O)}\right)
\leq C(r,\O) \|\tilde{p}\|_{W^{2,r}(\O)},
$$
where the first inequality follows from
Theorem 3.11 in \cite{MMY},
and the last inequality holds true since
the trace theorem. Then we have, by
 the Calderon-Zygmund inequality for $\tilde{p}$,
\begin{equation}\label{2.6}
\|\nabla \hat{p}\|_{L^{r}(\p\O)}
\leq C \|\dv\u\|_{L^{r}(\O)}\q \text{for } 1<r<\infty,
\end{equation}
where the constant $C$ depends on $r$ and
the Lipschitz character of $\O.$

Let $\hat{\w}$ be the weak solution of system \eqref{2.2}.
Then we introduce
$$
\hat{\mathbf{v}}(x)={\phi}(x)
-\zeta(x)
$$
with
$$
{\phi}(x)=\frac{1}{4\pi}\int_{\O}\frac{1}{|x-y|} \curl\u(y)dy,\q
\zeta(x)=\frac{1}{4\pi}\int_{\p\O}\frac{1}{|x-y|} \nu\times\u(y)dS_y.
$$
Using Green's formula, we have
$$
\aligned
\int_{\p\O}\frac{1}{|x-y|} \nu\times\u(y)dS_y&=\int_{\O}\frac{1}{|x-y|} \curl\u(y)dy
+\int_{\O}\nabla_y\left(\frac{1}{|x-y|}\right)\times\u(y)dy\\
&=\int_{\O}\frac{1}{|x-y|} \curl\u(y)dy
-\int_{\O}\nabla_x\left(\frac{1}{|x-y|}\right)\times\u(y)dy.
\endaligned
$$
The last integral of the above equality is divergence-free,
and hence we have $\dv{\hat{\mathbf{v}}}=0$ in $\O.$
By noting that $\Delta \zeta=0$ in $\O$, we then obtain
$$
\curl\curl \hat{\mathbf{v}}=\curl\u \q\text{in }\O.
$$
In the following, we establish the  estimate of $\curl\hat{\mathbf{v}}$.
From the trace theorem and the Calderon-Zygmund inequality, it follows that
$$
\|{\curl\phi}\|_{L^{r}(\p\O)}\leq C(r,\O)\|{\phi}\|_{W^{2, r}(\O)}
\leq C(r,\O) \|\curl \u\|_{L^{r}(\O)}.
$$
Therefore, it suffices to establish the  estimate of $\curl \zeta$.
Applying  Theorem 1.1 in \cite{Geng-Shen} again (since $\Delta \zeta=0$ in $\O$),  we have
$$
\|(\curl \zeta)^{*}\|_{L^{r}(\p\O)}
\leq C(r,\O) \left\|\frac{\p\zeta}{\p\nu}\right\|_{L^{r}(\p\O)}\q \text{for } 1<r<\infty.
$$
By the equality (see e.g. \cite{FA, FMM})
$$
\frac{\p\zeta}{\p\nu}=2\pi\nu\times\u
+\int_{\p\O}\frac{\p}{\p\nu(x)}\frac{1}{|x-y|} \nu\times\u(y)dS_y,
$$
then noting that we have,  from \cite[Theorem 1.0]{FA} and \cite{CMM},
$$
\left\|\int_{\p\O}\frac{\p}{\p\nu(x)}\frac{1}{|x-y|} \nu\times\u(y)dS_y\right\|_{L^{r}(\p\O)}
\leq C(r,\O) \left\|\nu\times\u\right\|_{L^{r}(\p\O)},
$$
we immediately obtain the estimate
$$
\|(\curl \zeta)^{*}\|_{L^{r}(\p\O)}
\leq C(r,\O) \left\|\nu\times\u\right\|_{L^{r}(\p\O)}.
$$
Combining with the estimate of
$\curl\phi$, we now get
\begin{equation}\label{2.7}
\|\curl \hat{\mathbf{v}}\|_{L^{r}(\p\O)}
\leq C \left(\|\curl \u\|_{L^{r}(\O)}
+\|\nu\times\u\|_{L^{r}(\p\O)}\right)\q \text{for } 1<r<\infty,
\end{equation}
where the constant $C$ depends on $r$ and
the Lipschitz character of $\O.$

Let $\hat{\mathbf{h}}=\hat{\w}-\hat{\mathbf{v}}.$ Then we have
$$
\aligned
\curl\curl\hat{\mathbf{h}}=0\q \text{and}\q\dv\hat{\mathbf{h}}=0 \q
&\text{in }\O,\\
 \nu\times\curl\hat{\mathbf{h}}=\nu\times(\u-\curl\hat{\mathbf{v}})\q &\text{on }\p\O.
\endaligned
$$
From the first equation, there exists a function $\hat{\varphi}$ with $\int_{\p\O} \hat{\varphi} dx=0$
such that $
\curl\hat{\mathbf h}=\nabla \hat{\varphi}$ in $\O.$
Then  from  the boundary condition, $\hat{\varphi}$ satisfies
$$
\Delta \hat{\varphi}=0\q  \text{in }\O;
\q \nabla_{\mathrm{tan}}\hat{\varphi}=\nu\times(\u-\curl\hat{\mathbf{v}})\q \text{on }\p\O.
$$
From Theorem 3.11 in \cite{MMY} we have, for $1< r<\infty,$
$$
\|(\nabla \hat{\varphi})^*\|_{L^{r}(\p\O)}\leq C(r,\O) \left(\|\nu\times(\u-\curl\hat{\mathbf{v}})\|_{L^{r}(\p\O)}\right).
$$
From \eqref{2.7} and the above inequality,  it follows that
$$
\|\curl\hat{\mathbf h}\|_{L^{r}(\p\O)}
\leq C(r,\O) \left(\|\curl\u\|_{L^{r}(\O)}
+\|\nu\times\u\|_{L^{r}(\p\O)}\right)\q \text{for } 1<r<\infty.
$$
Therefore, by \eqref{2.7} again  we have
\begin{equation}\label{2.8}
\|\curl\hat{\w}\|_{L^{r}(\p\O)}
\leq C \left(\|\curl\u\|_{L^{r}(\O)}
+\|\nu\times\u\|_{L^{r}(\p\O)}\right)\q \text{for } 1<r<\infty,
\end{equation}
where the constant $C$ depends on $r$ and
the Lipschitz character of $\O.$

 If we let
 $$
 \hat{\u}=\nabla \hat{p}+\curl\hat{\w},
 $$
then we have
 $$
\dv (\hat{\u}-\u)=0,\, \curl (\hat{\u}-\u)=0\,\,\,\text{in }\O,\q
\nu\times(\hat{\u}-\u)=0\,\,\,\text{on }\p\O.
$$
This gives $\hat{\u}=\u$ in $\O$.
Therefore,  the inequality \eqref{1.1}  holds true
since \eqref{2.6} and \eqref{2.8}.
Using Corollary 10.3(c) in \cite{MMP}, we finish our proof.
\end{proof}

\section{Proof of Theorem \ref{thm1.4}}\label{section4}

%\begin{Lem}\label{lem3.1}
%Let $\O$ be a bounded convex domain in $\mathbb R^3.$
%Let $\Psi\in L^r(\O)$ with $1<r<\infty.$ Then
%there exists a vector $\omega\in L^r(\O)$ such that
%$\curl\omega=\Psi,$ and we have the estimate
%$$
%\|\omega\|_{L^{r}(\O)}\leq C \|\curl\u\|_{L^{r}(\O)},
%$$
%where the constant C depends on $r$
%and the Lipschitz character of the domain $\O$.
%\end{Lem}
%\begin{proof}
%We solve the system
%$$
%\curl\omega=\Psi,\q   \dv\omega=0\q \text{in }\O,\q \nu\cdot\omega=0\q \text{on }\p\O.
%$$
%For $2\leq r<\infty,$ it follows from Theorem \ref{thm1.1}.
%Now let $1< r<2.$
%\end{proof}

We first prove a weak
reverse H\"{o}lder inequality near the boundary for a $\curl$-type system with the coefficient matrix
symmetric and
uniformly elliptic.

\begin{Lem}\label{lem4.2}
Let $\O$ be a bounded convex domain in $\mathbb R^3, $ and let the matrix
$A(x)$ be symmetric, bounded measurable,
uniformly elliptic and in $\mathrm{VMO}(\O).$
Let $Q\in\p\O$ and $0<s<s_0$ for some $s_0.$
Suppose that $\H$ satisfies
$$
\curl\left(A(x)\curl\H\right)=\curl((1-\varphi)\Psi) \q \text{and}\q \dv\H=0 \q \text{in }\O
$$
with the boundary condition $
\nu\times\H=0 $ on $\p\O$,
where $\Psi\in L^r(\O)$ and $\varphi\in C^{\infty}(\mathbb R^3)$
is a cut-off function
such that $\varphi=1$ on $B(Q, 4s)$ and $\varphi=0$ outside of $B(Q, 8s).$
Then for any $r>2$ we have
\begin{equation}\label{5.1}
\left\{\frac{1}{s^3}\int_{\O\bigcap B(Q,s)}|\curl\H|^{r} dx\right\}^{1/r}
\leq C\left\{\frac{1}{s^3}\int_{\O\bigcap B(Q,2s)}|\curl\H|^2 dx\right\}^{1/2},
\end{equation}
where the constant $C$ depends on $r, s_0$ and the Lipschitz character of $\O.$
\end{Lem}

\begin{proof} From the assumptions, we have
$\curl\left(A(x)\curl\H\right)=0$ in $B(Q, 4s)$.
Thus there exists a function $\phi$ defined on
$B(Q, 4s)$ such that
\begin{equation}\label{5.7}
A(x)\curl\H=\nabla \phi.
\end{equation}
Then $\phi$ satisfies
$$
\dv(A^{-1}(x)\nabla\phi)=0\q \text{in }B(Q, 4r)\bigcap \O,\q
\nu\cdot(A^{-1}(x)\nabla\phi)=0\q \text{on }B(Q, 4r)\bigcap\p\O.
$$
Based on
Theorem 2.1 in \cite{Geng-Shen},
then from Lemma 4.1, Lemma 4.2 and Theorem 2.1 in \cite{Geng},
it follows that
$$
\left\{\frac{1}{s^3}\int_{\O\bigcap B(Q,s)}|\nabla\phi|^{r} dx\right\}^{1/r}
\leq C\left\{\frac{1}{s^3}\int_{\O\bigcap B(Q,2s)}|\nabla\phi|^2 dx\right\}^{1/2}.
$$
By applying the inequality
$$
\frac{1}{\Lambda^r}|\nabla\phi|^{r} \leq |A^{-1}(x)\nabla\phi|^{r}
\leq \frac{1}{\lambda^r}|\nabla\phi|^{r},
$$
and then using \eqref{5.7}, we immediately get \eqref{5.1}.
\end{proof}

We now give the proof of Theorem \ref{thm1.4}.

\begin{proof}[Proof of Theorem \ref{thm1.4}]
We  decompose
$$\u=\mathbf u_1+\nabla u_2+\mathbf u_3,
$$
where
$\mathbf u_1, u_2, \mathbf u_3$ are to be determined.

Step 1. Construct $\mathbf u_1.$
Consider the following Neumann problem
$$
\dv(A^{-1}(x)\nabla\phi)=0\q \text{in }\O,\q
\nu\cdot(A^{-1}(x)\nabla\phi)=\mathrm{Div}\,\mathbf g\q \text{on }\p\O.
$$
This problem studied by Geng in \cite{Geng} is solvable in Lipschitz domains
if $\mathrm{Div}\,\mathbf g\in B^{-1/r,r}(\p\O)$
with $3/2-\epsilon<r<3+\epsilon,$ see \cite[Lemma 5.2]{Geng}. To prove this,
it suffices to establish a weak
reverse H\"{o}lder inequality
$$
\left(\frac{1}{s^3}\int_{B(x_0,s)\bigcap \O}|\nabla v|^r dx\right)^{\frac{1}{r}}
\leq C_0\left(\frac{1}{s^3}\int_{B(x_0,2s)\bigcap \O}|\nabla v|^2 dx\right)^{\frac{1}{2}}
$$
for any $0<s<s_0$ ($s_0$ depends on the domain) and
any $v\in W^{1,2}(B(x_0,2s)\bigcap \O)$ satisfying
the above Neumann problem in $B(x_0,2s)\bigcap \O$
with the boundary condition $\mathrm{Div}\,\mathbf g=0$
on $B(x_0,2s)\bigcap \p\O,$ see Theorem 1.1 and Lemma 5.1 in \cite{Geng}.
For Lipschitz domains, the weak
reverse H\"{o}lder inequality only
holds for $2<r\leq 3+\epsilon$ (see \cite[Lemma 4.1]{Geng}).
However, for any convex domains,
the range of the index $r$ can be extended to $2<r<\infty$, which may be proved by
applying Theorem 2.1 in \cite{Geng-Shen} to
Lemma 4.1, Lemma 4.2 and Theorem 2.1 in \cite{Geng}.
Based on this, the conclusion of Lemma 5.2
in \cite{Geng} can be obtained for any $1<r<\infty$ if the domain $\O$ is convex.
That is, the above Neumann problem  is solvable for any  $1<r<\infty$, and
we can deduce the estimate
$$
\|\nabla\phi\|_{L^r(\O)}\leq C \|\mathrm{Div}\,\mathbf g\|_{B^{-1/r,r}(\p\O)},
$$
where the constant C depends on $r$ and the Lipschitz character of $\O$.

We now solve the following div-curl system
$$
\curl\mathbf u_1=A^{-1}(x)\nabla\phi,\q   \dv\mathbf u_1=0\q \text{in }\O,\q \nu\times\mathbf u_1=\mathbf g\q \text{on }\p\O.
$$
By the proof of Theorem 10.1 in \cite{MMP}, we can
conclude that there exists a unique solution in $L^r(\O)\bigcap L^2(\O)$ space
to this system.  Applying Theorem \ref{thm1.2}, we have $\nu\cdot\u\in L^{r}(\p\O)$
and the estimate \eqref{1.1} holds. From the integral representation formula
for vector fields (see \cite[Theorem 3.2]{MMP}) and recalling that
$A(x)$ is positive, then we obtain the estimate
$$
\|\mathbf u_1\|_{L^r(\O)}\leq C(r,\O)\left(\|\nabla\phi\|_{L^r(\O)}
+\|\mathbf g\|_{L^{r}(\p\O)}\right),
$$
see the estimate
of $\zeta(x)$ in the proof of Theorem \ref{thm1.2}
or we may use Corollary 10.3(c) in \cite{MMP}. Combining with the estimate for $\nabla\phi$ and by
the first equation in the div-curl system, we immediately get
$$\|\mathbf u_1\|_{L^r(\O)}+\|\curl\mathbf u_1\|_{L^r(\O)}\leq
C\left( \|\mathrm{Div}\,\mathbf g\|_{B^{-1/r,r}(\p\O)}+\|\mathbf g\|_{L^{r}(\p\O)}\right),
$$
where the constant C depends  on $r$ and the Lipschitz character of $\O$.

Step 2. Construct $u_2.$ By Theorem 1.3 in \cite{Geng-Shen}, we
take the Helmholtz decomposition to $\mathbf F$ and to  $\mathbf u_1:$
$$
\mathbf F=\nabla p_{\mathbf F}+\curl \w_{\mathbf F},\q
\mathbf u_1=\nabla p_{\mathbf u_1}+\curl \w_{\mathbf u_1}
$$
Let $u_2\in W_0^{1,r}(\O)$ be the weak solution of the form
$$
\int_{\O}\nabla u_2\cdot\nabla\psi=
\int_{\O}\left(\nabla p_{\mathbf F}-\nabla p_{\mathbf u_1}\right)
\cdot\nabla\psi\q \text{for any }\psi\in W_0^{1,r/(r-1)}(\O).
$$
Then there exists a constant C depending on $r$
and the Lipschitz character of $\O$ such that (see e.g. \cite{JW})
$$
\|\nabla u_2\|_{L^r(\O)}\leq C \|\nabla p_{\mathbf f}-\nabla p_{\mathbf u_1}\|_{L^{r}(\O)}
\leq C\left(\|\mathbf F\|_{L^r(\O)}+\|\mathbf u_1\|_{L^r(\O)}\right),
$$
where the last inequality follows from Theorem 1.3 in \cite{Geng-Shen}.

Step 3. Construct $\mathbf u_3.$ Consider the system
\begin{equation}\label{4.3}
\aligned
\curl\left(A(x)\curl\mathbf u_3\right)+\mathbf u_3&=\mathbf F-\mathbf u_1-\nabla u_2+\curl\mathbf f\q
&\text{in }\O,\q\\
\nu\times\mathbf u_3&=0\q &\text{on }\p\O.~
\endaligned
\end{equation}
Now we have $\dv(\mathbf F-\mathbf u_1-\nabla u_2)=0$ in $\O.$
By Poincar\'{e}'s lemma (see \cite[p.214]{DL}),  there exists a vector $\omega\in L^r(\O)$
such that $\curl\omega=\mathbf F-\mathbf u_1-\nabla u_2$ and $\omega$ satisfies the estimate
\begin{equation}\label{4.4}
\|\omega\|_{L^r(\O)}\leq C(r,\O)\left(\|\mathbf F\|_{L^r(\O)}
+\|\mathbf u_1\|_{L^r(\O)}+\|\nabla u_2\|_{L^r(\O)}\right).
\end{equation}
To obtain the existence of $\mathbf u_3$, we first assume $r\geq 2.$
From the Lax-Milgram Lemma, it follows that $\mathbf u_3\in H^1(\O).$
For $1<r<2$, it is necessary to establish the a priori estimate for $\mathbf u_3,$
then take the usual approximation argument to obtain the existence.

We now give the estimate for $\mathbf u_3.$ Note that $\mathbf u_3\in L^6(\O)$
by the imbedding theorem.
By Poincar\'{e}'s lemma again, there exists
a vector $\mathbf \psi\in L^6(\O)$ such that $\mathbf u_3=\curl\mathbf\psi.$
Actually, by Theorem 1.3 in \cite{Geng-Shen}
we can further let $\mathbf \psi$ satisfy
 $\dv\mathbf\psi=0$ in $\O$ and $\nu\cdot\mathbf\psi=0$
on $\p\O$. From Corollary \ref{cor2.4},  we have the estimate
$$
\|\mathbf\psi\|_{L^{\infty}(\O)}\leq C \|\mathbf u_3\|_{L^{3,1}(\O)}.
$$
Since $H^1(\O)$ is continuously imbedded into
the Lorentz space $L^{3,1}(\O)$ and by
$H^1$ estimate for $\u_3$, we can obtain
\begin{equation}\label{4.5}
\|\mathbf\psi\|_{L^r(\O)}\leq C\|\mathbf\psi\|_{L^{\infty}(\O)}\leq C \|\mathbf u_3\|_{H^1(\O)}
\leq C \left(\|\omega\|_{L^2(\O)}+\|\mathbf f\|_{L^2(\O)}\right),
\end{equation}
where the constants C depend on $r$ and the Lipschitz character of $\O$.

Let $\Psi=\omega+\mathbf f-\mathbf\psi.$ Then
$\mathbf u_3$ satisfies the system
$$
\curl\left(A(x)\curl\mathbf u_3\right)=\curl\Psi,\q \dv\mathbf u_3=0\q  \text{in }\O,\q \nu\times\mathbf u_3=0\q \text{on }\p\O.
$$
Based on the  weak
reverse H\"{o}lder inequality (Lemma \ref{lem4.2}),
the proof of Theorem 1.1 in \cite{Geng} with the $\nabla$ operator
replaced by the $\curl$ operator is also applicable. Thus,  we can
deduce that
$$
\|\curl\mathbf u_3\|_{L^r(\O)}\leq C(r, \O) \|\Psi\|_{L^r(\O)}.
$$
Therefore, by Theorem \ref{thm1.2} (as the estimate of $\u_1$)
we have that
$$
\|\mathbf u_3\|_{L^r(\O)}+\|\curl\mathbf u_3\|_{L^r(\O)}
\leq C(r, \O) \|\Psi\|_{L^r(\O)}
$$
Since $\Psi=\omega+\mathbf f-\mathbf\psi$ and the
estimate \eqref{4.5} on $\psi$, we then get
$$
\|\mathbf u_3\|_{L^r(\O)}+\|\curl\mathbf u_3\|_{L^r(\O)}
\leq C(r,\O) \left(\|\omega\|_{L^r(\O)}+\|\mathbf f\|_{L^r(\O)}\right).
$$
From \eqref{4.4}, we now have
\begin{equation}\label{11.5}
\aligned
&\|\mathbf u_3\|_{L^r(\O)}+\|\curl\mathbf u_3\|_{L^r(\O)}\\
\leq &C(r,\O) \left(\|\mathbf F\|_{L^r(\O)}
+\|\mathbf u_1\|_{L^r(\O)}+\|\nabla u_2\|_{L^r(\O)}+\|\mathbf f\|_{L^r(\O)}\right),
\endaligned
\end{equation}
Plugging the estimates of $\u_1$ (step 1) and of $\nabla u_2$ (step 2)
back to the above inequality, then noting
that $\u=\mathbf u_1+\nabla u_2+\mathbf u_3$,
we finally obtain that, for $2\leq r<\infty$,
\begin{equation}\label{4.6}
\aligned
&\|\mathbf u\|_{L^r(\O)}+\|\curl\mathbf u\|_{L^r(\O)}\\
\leq &C \left(\|\mathbf F\|_{L^r(\O)}
+\|\mathbf f\|_{L^r(\O)}+\|\mathrm{Div}\,\mathbf g\|_{B^{-1/r,r}(\p\O)}+\|\mathbf g\|_{L^{r}(\p\O)}\right),
\endaligned
\end{equation}
where the constant C depends only on $r$ and the Lipschitz character of $\O$.

To obtain the  a priori estimate for $\mathbf u_3$ if $1< r<2,$
we take the duality argument.
For any given vector  $\mathbf G\in L^{r/(r-1)}(\O),$
we solve the following system
$$
\curl (A(x)\curl\mathbf v)+\mathbf v=\mathbf G\q
\text{in }\O,\q \nu\times\mathbf v=0\q \text{on }\p\O.
$$
From \eqref{4.6}, we have the estimate for $\mathbf v:$
\begin{equation}\label{4.7}
\|\mathbf v\|_{L^{r/(r-1)}(\O)}+\|\curl\mathbf v\|_{L^{r/(r-1)}(\O)}
\leq C(r,\O) \|\mathbf G\|_{L^{r/(r-1)}(\O)}.
\end{equation}
Let $\langle\cdot, \cdot\rangle$ denote the duality pairing between
$L^r(\O)$ and $L^{r/(r-1)}(\O).$ Since $A(x)=A^{T}(x),$ we have
$$
\langle\mathbf u_3, \mathbf G\rangle
=\langle\mathbf u_3, \curl (A(x)\curl\mathbf v)+\mathbf v\rangle
=\langle\curl (A(x)\curl\mathbf u_3)+\mathbf u_3, \mathbf v\rangle.
$$
From \eqref{4.3},  it follows that
$$
\langle\mathbf u_3, \mathbf G\rangle
=\langle\curl (\omega+\mathbf f), \mathbf v\rangle
=\langle \omega+\mathbf f, \curl \mathbf v\rangle
$$
Combining with \eqref{4.7}, we have
$$
\|\mathbf u_3\|_{L^r(\O)}
\leq C \left(\|\omega\|_{L^r(\O)}+\|\mathbf f\|_{L^r(\O)}\right).
$$
To obtain the estimate for $\curl\mathbf u_3$,
we solve the following system
$$
\curl (A(x)\curl\mathbf m)+\mathbf m=\curl\mathbf h\q
\text{in }\O,\q \nu\times\mathbf m=0\q \text{on }\p\O
$$
for any given vector  $\mathbf h\in L^{r/(r-1)}(\O).$
From \eqref{4.6}, we have the estimate for $\mathbf m:$
\begin{equation}\label{10.8}
\|\mathbf m\|_{L^{r/(r-1)}(\O)}+\|\curl\mathbf m\|_{L^{r/(r-1)}(\O)}
\leq C(r,\O) \|\mathbf h\|_{L^{r/(r-1)}(\O)}.
\end{equation}
Then
$$
\langle\curl\mathbf u_3, \mathbf h\rangle
=\langle\mathbf u_3, \curl \mathbf h\rangle
=\langle \omega+\mathbf f, \curl \mathbf m\rangle.
$$
This shows the estimate, by \eqref{10.8},
$$
\|\curl\mathbf u_3\|_{L^r(\O)}
\leq C(r,\O) \left(\|\omega\|_{L^r(\O)}+\|\mathbf f\|_{L^r(\O)}\right).
$$
Therefore, for any $1<r<\infty,$ we always have the estimate
\eqref{11.5}.

From step 1-step 3, we now have the inequality \eqref{1.5}.
The uniqueness is obvious since \eqref{1.5}.
If  $\dv\mathbf F\in L^r(\O)$, then $\dv\mathbf u\in L^r(\O).$
It follows from Theorem \ref{thm1.2} that we have
$\u\in L_{1/r}^r(\O)$ if $2\leq r<\infty$
and $\u\in B_{1/r}^{r, 2}(\O)$ if $1<r<2.$
We end our proof.
\end{proof}

Finally, we consider the Maxwell-type system \eqref{1.10}
in Lipschitz domains. For simplicity, we let $\mathbf g=0$.
This system was studied in
the space $H^{s,r}_0 (curl; \O)$ by Kar and Sini, see \cite{KS}.
When  $s=0$, they gave a condition that
 characterizes the range of $r$ such that the problem is well-posed,
 see Remark 2.2 in \cite{KS}.
However, we may notice that by this condition it is not easy to
check how large the range for $r$ is.

Based on Lemma \ref{lem3.3} below and the proof of Theorem \ref{thm1.4},
we say, to show the well-posedness of this problem,
 the condition given by Kar and Sini is not needed if the coefficient matrix
$A(x)$ is symmetric, bounded measurable,
uniformly elliptic and in $\mathrm{VMO}(\O).$

Denote by
$$
I=\left\{r~:~ \frac{2}{3}\left(1-\frac{1}{p_{\O}}\right)<\frac{1}{r}<\frac{1}{3}\left(\frac{2}{p_{\O}}+1\right)\right\},
$$
where $p_{\O}$ is determined by the Lipschitz character  of the
 domain $\O$, see \cite{MM}.

\begin{Thm}\label{thm3.2}
Let $\O$ be a bounded Lipschitz domain in $\mathbb{R}^3$.
Assume that the coefficient matrix
$A(x)$ is symmetric, bounded measurable,
uniformly elliptic and in $\mathrm{VMO}(\O).$
Suppose that $\mathbf F\in L^r(\O)$ and
$\mathbf f\in L^r(\O)$  with $r\in I$,
then
there exists a unique solution $\u\in H^r(\curl,\O)$
of system \eqref{1.10} with  $\mathbf g=0$,
and the solution $\u$ satisfies the estimate
$$
\|\mathbf u\|_{L^r(\O)}+\|\curl\mathbf u\|_{L^r(\O)}
\leq C \left(\|\mathbf F\|_{L^r(\O)}+\|\mathbf f\|_{L^r(\O)}\right),
$$
where the constant $C$ depends on $r$ and
the Lipschitz character  of $\O.$
\end{Thm}
\begin{proof}
The proof is quite similar to that of Theorem \ref{thm1.4}, we here omit it.
\end{proof}

\begin{Lem}\label{lem3.3}
Let $\O$ be a Lipschitz domain and let $r\in I.$ For any $\mathbf\psi\in L^3(\O)$
with $\dv\psi=0$ in $\O,$
there exists
a vector $\omega\in L^r(\O)$ with $\dv\omega=0$ in $\O$ and $\nu\cdot\omega=0$
on $\p\O$ such that $\psi=\curl\omega$, and we have the estimate
\begin{equation}\label{4.8}
\|\omega\|_{L^r(\O)}\leq C \|\psi\|_{L^3(\O)},
\end{equation}
where  the constant $C$ depends on $r$ and
the Lipschitz character of $\O.$
\end{Lem}
\begin{proof}
It suffices to show that the
inequality \eqref{4.8} holds for $r>3.$
The method of our proof goes back to \cite{Cos}.
As in \cite{Cos}, take $R$ sufficiently large such that $\O\subset B_R.$
Let $\chi$ be the solution of the equation
$$
\Delta \chi=0 \q  \text{in }B_R\backslash\O;
\q \frac{\p \chi}{\p \nu}=\nu\cdot \mathbf \psi \q  \text{on }\p\O;
\q \frac{\p \chi}{\p \nu}=0 \q  \text{on }\p B_R.
$$
It follows that
$$
\|\nabla \chi\|_{L^3(B_R\backslash\O)}
\leq C(\O) \|\nu\cdot \mathbf \psi\|_{B^{-1/3,3}(\p\O)}
\leq C(\O) \|\mathbf \psi\|_{L^{3}(\O)}.
$$
Let
$$
\mathbf f=\mathbf \psi \q  \text{in }\O; \q
\mathbf f=\nabla \chi \q  \text{in }B_R\backslash\O;\q
\mathbf f=0 \q  \text{in }\mathbb{R}^3\backslash B_R.
$$
Then we have
$$
\dv\mathbf f=0\q  \text{in the sense of distribution in }\mathbb{R}^3.
$$
Denote by
$$
\mathbf v=\curl\int_{\mathbb{R}^3}\frac{1}{|x-y|} \mathbf f(y) dy.
$$
By the Calderon-Zygmund inequality we have
\begin{equation}\label{4.9}
\aligned
\|\mathbf v\|_{L^{r}(\O)}\leq C\|\mathbf v\|_{W^{1,3}(B_R)}\leq C \|\mathbf f\|_{L^{3}(\mathbb{R}^3)}
\leq &C \left(\|\mathbf \psi\|_{L^{3}(\O)}+\|\nabla \chi\|_{L^3(B_R\backslash\O)}\right)\\
\leq&
C \|\psi\|_{L^{3}(\O)},
\endaligned
\end{equation}
where the constants C depend only on the Lipschitz character of $\O$.

Introduce $\mathbf h=\omega-\mathbf v$. Then
$$
\curl\mathbf h=0,\q \dv\mathbf h=0 \q  \text{in }\O,
\q \nu\cdot\mathbf h=-\nu\cdot\mathbf v\q \text{on }\p\O.
$$
Thus there exists a function $\varphi$ such that $
\mathbf h=\nabla \varphi$ in $\O.$
This gives that
$$
\Delta \varphi=0\q  \text{in }\O;
\q \frac{\p \varphi}{\p \nu}=-\nu\cdot\mathbf v\q \text{on }\p\O.
$$
Therefore,
$$
\|\mathbf h\|_{L^{r}(\O)}=\|\nabla \varphi\|_{L^{r}(\O)}\leq C (r, \O)\|\nu\cdot\mathbf v\|_{B^{-1/r,r}(\p\O)}
\leq C(r, \O) \|\mathbf v\|_{L^{r}(\O)}.
$$
Combining with \eqref{4.9}, we obtain the inequality \eqref{4.8}.
We end our proof.
\end{proof}

\v0.2in

\noindent{\bf ACKNOWLEDGMENTS}
\v0.1in
The author is grateful to his supervisor,
Professor Xingbin Pan, for guidance and
constant encouragement.
The work was supported by the
National Natural Science Foundation of China
grant No. 11771135, 11671143.

\appendix{}

\section{Proof of inquality \eqref{3.6}}

In this section we give the proof of \eqref{3.6}
in Lemma \ref{lem3.2} if $\u\in C^3({\O})\bigcap C^2(\bar{\O})$
and  the domain $\O$ is smooth. The proof dues to  Cianchi and Maz'ya
(see \cite{CM}).

\begin{proof}
Introduce the distribution function of $v$ (see \cite{CM}):
$$\mu_{v}(t)=\mu(\{|v|>t\}),\q t>0,
$$
and the nonincreasing rearrangement of $v$:
\begin{equation}\label{mm2.6}
v^{*}(s)=\sup\{t>0, \mu_{v}(t)> s\},\q t>0.
\end{equation}

By the isoperimetric inequality and the coarea formula (see \cite[Lemma 5.2]{CM}),
for $t\geq |\H|^*(|\Omega|/2)$ we can obtain that
\begin{equation}\label{2.15}
\left(\int_{\{|\H|=t\}}
|\nabla|\H||dS\right)^{-1}\leq  C \left(-\mu_{|\H|}'(t)\right)\mu_{|\H|}^{-\frac{4}{3}}(t),
\end{equation}
where the constant $C$ depends on the Lipschitz character of $\O.$
Note that (see the inequality (6.38) in \cite{CM})
$$
\int_{\{|\H|=t\}}|\curl\u| dS
\leq \left(-\frac{d}{dt}\int_{\{|\H|>t\}} |\curl\u|^2 dx\right)^{1/2}
\left(\int_{\{|\H|=t\}} |\nabla |\H|| dS\right)^{1/2}.
$$
Denote by $t_0:=|\H|^*(|\Omega|/2).$
Then for any $T$ satisfying $t_0<T<\|\H\|_{L^{\infty}(\O)}$ we have
$$
\aligned
&\int_{t_0}^{T}\left(-\frac{d}{dt}\int_{\{|\H|>t\}} |\curl\u|^2 dx\right)^{1/2}
\left(\int_{\{|\H|=t\}} |\nabla |\H|| dS\right)^{-1/2} dt\\
\leq & C(\O) \int_{t_0}^{T}\left(\mu_{|\H|}'(t)
\frac{d}{dt}\int_{\{|\H|>t\}} |\curl\u|^2 dx\right)^{1/2}
\mu_{|\H|}^{-\frac{2}{3}}(t)dt.
\endaligned
$$
Since
$$
\aligned
&\int_{t_0}^{T}\left(\mu_{|\H|}'(t)
\frac{d}{dt}\int_{\{|\H|>t\}} |\curl\u|^2 dx\right)^{1/2}
\mu_{|\H|}^{-\frac{2}{3}}(t)dt\\
\leq&
\int_{0}^{|\O|}\left(
\frac{d}{ds}\int_{\{|\H|>|\H|^*(s)\}} |\curl\u|^2 dx\right)^{1/2}
s^{-\frac{2}{3}}ds
\endaligned
$$
and from \cite[Proposition 3.4, Lemma 3.5]{CM1} (also see \cite{CM}), we have
$$
\int_{0}^{|\O|}\left(
\frac{d}{ds}\int_{\{|\H|>|\H|^*(s)\}} |\curl\u|^2 dx\right)^{1/2}
s^{-\frac{2}{3}}ds\leq C(\O)
\|\curl\u\|_{L^{3,1}(\O)}.
$$
Then we obtain that
\begin{equation}\label{m2.15}
\int_{t_0}^{T} \left(\int_{\{|\H|=t\}}
|\nabla|\H||dS\right)^{-1}\int_{\{|\H|=t\}}|\curl\u| dS dt
\leq  C(\O) \|\curl\u\|_{L^{3,1}(\O)}.
\end{equation}
Similarly, we have (\cite[Lemma 3.6]{CM1})
\begin{equation}\label{m2.16}
\int_{t_0}^{T}\left(\int_{\{|\H|=t\}}
|\nabla|\H||dS\right)^{-1} \int_{\{|\H|>t\}} |\curl\u|^2 dx
dt\leq  C(\O) \|\curl\u\|_{L^{3,1}(\O)}^2,
\end{equation}
where the constants $C$ in \eqref{m2.15}
and \eqref{m2.16} depend on the Lipschitz character of $\O.$

Therefore, from \eqref{3.3}, \eqref{m2.15} and \eqref{m2.16} we have
$$
T^2-t_0^2\leq C T\|\curl\u\|_{L^{3,1}(\O)}+ C  \|\curl\u\|_{L^{3,1}(\O)}^2.
$$
Note that
$$
t_0:=|\H|^*(|\Omega|/2)
\leq \frac{2}{|\O|}\int_{\O} |\H| dx\leq
 C(\O)\|\curl\u\|_{L^{3,1}(\O)}.
$$
Then we have
$$
T\leq C(\O) \|\curl\u\|_{L^{3,1}(\O)}.
$$
Now letting $T\to\|\H\|_{L^{\infty}(\O)}$, we obtain that
\begin{equation}\label{3.11}
\|\curl\w\|_{L^{\infty}(\O)}\leq C(\O) \|\curl\u\|_{L^{3,1}(\O)}.
\end{equation}
We end our proof.
\end{proof}

\end {document}